\documentclass[preprint,11pt, 2014]{elsarticle}

\usepackage[all]{xy}
\usepackage{amsmath,amssymb,amsthm}
\usepackage{graphicx,epsfig}
\usepackage{enumerate}
\usepackage{tikz}
\usepackage{ytableau}

\numberwithin{equation}{section}

\swapnumbers
\theoremstyle{plain}
\newtheorem{theorem}{Theorem}[section]

\newtheorem{proposition}[theorem]{Proposition}

\theoremstyle{definition}

\newtheorem{remark}[theorem]{Remark}
\newtheorem{definition}[theorem]{Definition}

 \definecolor{light-gray}{gray}{0.96}

\journal{X}

\title{Convex Domino Towers}
\author{Tricia Muldoon Brown\\
Department of Mathematics\\
Armstrong State University\\
11935 Abercorn Street\\
 Savannah, GA 31419\\
 USA\\
 patricia.brown@armstrong.edu}

\begin{document}

\begin{abstract}
We study convex domino towers using a classic dissection technique on polyominoes to find the generating function and an asymptotic approximation.

\end{abstract}

\maketitle

\section{Introduction}
If you have ever had the pleasure to play with blocks like LEGO\textsuperscript{\textregistered}s or DUPLO\textsuperscript{\textregistered}s with young children, you will see that their first instinct is to stack the blocks into towers.  This provokes the enumerative question: how many different towers can be built?  In previous work~\cite{Brown}, this author restricted to blocks of fixed size, with unit width, unit height, and a length of $k$ units, called $k$-omino blocks and enumerated $k$-omino towers in terms of hypergeometric functions.  In this manuscript, we will use the classical method of dissection to study convex domino (2-omino) towers.

A {\it domino block} is a $2$-omino block which is two units in length and has two \textit{ends}, a \textit{left end} and a \textit{right end}.  More generally, a \textit{polyomino} is a collection of unit squares with incident sides.   We note, a domino can refer to both the three-dimensional domino block and the two-dimensional domino polyomino through the correspondence of the block to the polyomino described by the boundary of the length two vertical face of the block.  Similarly, the boundary of the vertical face of a collection of stacked domino blocks may also define a fixed polyomino.  In such a collection, a domino is in the \textit{base} if no dominoes are underneath it, and the \textit{level} of a domino will be its vertical distance from the base.

Under the name of dissection problems, some of the first published results on polyominoes appeared in {\it The Fairy Chess Review} in the 1930's and 1940's and were named and popularized by Solomon Golomb~\cite{Golomb_AMM, Golomb} and Martin Gardener~\cite{Gardener}.  Researchers from various disciplines including mathematicians, chemists, and physicists have been interested in polyominoes and their applications.  Traditionally, polyomino problems have been easy to describe, but often difficult to solve.  For example, no formula is known for the number of polyominoes parametrized by area, although asymptotic bounds~\cite{Klarner_Rivest} and numerical estimates~\cite{Jensen} do exist.  This general case may be unknown, but enumerative results are available for certain classes of polyominoes using parameters such as area, perimeter, length of base or top, number of rows or columns, and others. 

From the enumerative viewpoint, here we only consider \textit{fixed polyominoes}, also called a fixed animals, which are oriented polyominoes such that different orientations of the same free shape are considered distinct.  We define a class of such fixed polyominoes called domino towers in terms of their area $2n$ and base of length $2b$ as follows:
\begin{definition} For $n\geq b\geq 1$, an \textit{(n,b)-domino tower} is a fixed polyomino created by sequentially placing $n-b$ dominoes horizontally on a convex, horizontal base composed of $b$ dominoes, such that if a non-base domino is placed in position indexed by $\{(x,y),(x+1,y)\}$, then there must be a domino in position $\{(x-1,y-1), (x,y-1)\}$, $\{(x,y-1), (x+1,y-1)\}$, or $\{(x+1,y-1), (x+2,y-1)\}$.
\end{definition}
Some enumerative results on domino towers are known.  Specifically, the number of $(n,b)$-domino towers is given by $2n-1\choose n-b$ for $n\geq b\geq 1$ and the number of $n$-domino towers is $4^{n-1}$ for $n\geq 1$, see Brown~\cite{Brown}.  Here, we wish to enumerate the class of {\it convex} domino towers.  To define convexity, we start with some vocabulary associated with polyominoes.  A \textit{column} of a polyomino is the intersection of the polyomino with an infinite vertical line of unit squares.  Similarly, a \textit{row} of a polyomino is the intersection of the polyomino with an infinite horizontal line of unit squares.  A polyomino is said to be \textit{column-convex}, respectively \textit{row-convex}, if all its columns, respectively rows, are convex.  Finally, a polyomino is \textit{convex} if it is both column- and row-convex.  Figure~\ref{convex_fig} displays examples of convex $(18,4)$-domino towers.

\begin{figure}[h]
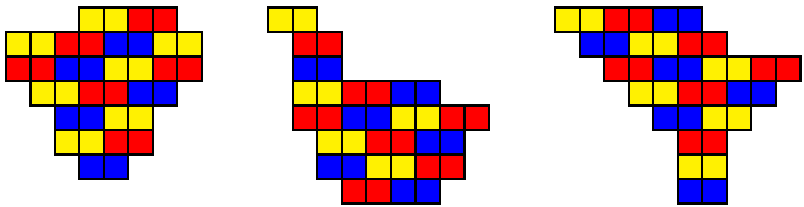

\centering
  \ytableausetup{smalltableaux}
  \begin{ytableau}
   \none &   \none   & \none &*(yellow) & *(yellow) &*(red) & *(red) \\
 *(yellow) & *(yellow)&*(red) & *(red) &*(blue) & *(blue) &*(yellow) & *(yellow)\\
   *(red) & *(red)   &*(blue) & *(blue)&*(yellow) & *(yellow)&*(red) & *(red)\\
   \none &  *(yellow) & *(yellow) &*(red) & *(red)  &*(blue) & *(blue)\\
   \none &   \none & *(blue) & *(blue)  &*(yellow) & *(yellow)\\
   \none &  \none &  *(yellow) & *(yellow)&*(red) & *(red)\\
   \none &  \none & \none & *(blue) & *(blue) 
  \end{ytableau}
\qquad
\begin{ytableau}
*(yellow) & *(yellow) \\
\none &*(red) & *(red) \\
\none &*(blue) & *(blue) \\
 \none &*(yellow) & *(yellow)&*(red)&*(red) & *(blue) & *(blue) \\
 \none &  *(red) & *(red)   &*(blue) & *(blue)&*(yellow) & *(yellow)&*(red) & *(red)\\
 \none & \none &*(yellow) & *(yellow) &*(red) & *(red)  &*(blue) & *(blue)\\
 \none & \none & *(blue) & *(blue)&*(yellow) & *(yellow) &*(red) & *(red) \\
 \none & \none &   \none & *(red) & *(red)  &*(blue) & *(blue)
  \end{ytableau}
\qquad
  \begin{ytableau}
*(yellow) & *(yellow) &*(red) & *(red) &*(blue) & *(blue)\\
\none&*(blue) & *(blue) & *(yellow) & *(yellow)&*(red) & *(red) \\
  \none &\none   & *(red) & *(red)   &*(blue) & *(blue)&*(yellow) & *(yellow)&*(red) & *(red)\\
 \none & \none & \none &  *(yellow) & *(yellow) &*(red) & *(red)  &*(blue) & *(blue)\\
 \none & \none & \none &   \none & *(blue) & *(blue)  &*(yellow) & *(yellow)\\
 \none & \none & \none &  \none &  \none & *(red) & *(red)\\
 \none & \none & \none &  \none & \none & *(yellow) & *(yellow)\\
 \none & \none & \none &  \none & \none & *(blue) & *(blue) 
  \end{ytableau}
\caption{Examples of convex $(18,4)$-domino towers}
\label{convex_fig}
\end{figure}

Historically classes of column-convex, row-convex, and convex polyominoes have been of interest.  Klarner and Rivest~\cite{Klarner_Rivest} and Bender~\cite{Bender} give asymptotic formulas for the number of convex polyominoes by area.  Klarner~\cite{Klarner} describes the generating function for the number of column-convex polyominoes by area, Delest~\cite{Delest} and Delest and Viennot~\cite{Delest_Viennot}, respectively, find generating functions for column-convex with parameters area with number of columns and for convex polyominoes with perimeter $2n$, respectively, and Domo\c{c}os~\cite{Domocos} gives a generating function for the number of column-convex polyominoes by area, number of columns, and number of pillars.  

This notion of convexity can be applied to the polyomino class of domino towers.  Following the examples of Klarner and Rivest~\cite{Klarner_Rivest}, Bender~\cite{Bender}, and Wright~\cite{Wright}, the strategy applied in this paper is to dissect a convex tower into more-easily understood distinct domino towers.   Section~\ref{convex_towers} will give the dissection of a convex domino towers into an upper and lower domino tower.  The generating function enumerating the number of convex $(n,b)$-domino towers, $C_b(z)$ is stated in terms of generating functions of the dissected  pieces.  In Section~\ref{asymptotic_behavior}, we describe the asymptotic behavior of convex $(n,b)$-domino towers with the result that the coefficients are of exponential order $2^n$, that is, 
\begin{equation*}
[z]^n C_b(z) \sim \theta_b(n) 2^n
\end{equation*}
where 
\begin{equation*}
\theta_b(n) \rightarrow (1/2)^{b-1} \prod_{k=1}^{b-1} \frac{2^k}{2^k-1} \sim 3.46(1/2)^{b-1}
\end{equation*}
as $b>>0$.  
The final section, Section~\ref{conclusion}, concludes with a few questions and remarks.

%%%%%%%%%%%%%%%%%%%%%%%%%%%%%%%%%%%%%%%%%%%%%%%%%%%%%%%%%%%%%%%%%%%%%%%%%%%%
\section{Convex domino towers}\label{convex_towers}

In order to find a generating function for the number of convex domino towers, we first introduce a well-studied class of convex polyominoes called  parallelogram polyominoes.  A \textit{parallelogram polyomino} is a polyomino given by a pair of lattice paths $(\tau, \sigma)$ such that $\tau$ begins with a vertical step, $\sigma$ begins with a horizontal step, and $\tau$ and $\sigma$ intersect exactly once at their endpoints.  The set of parallelogram polyominoes enumerated by perimeter is in bijection with the set of Dyck words and hence is enumerated by Catalan numbers, see P\'olya~\cite{Polya} and Gessel and Viennot~\cite{Gessel_Viennot}.  A subset of parallelogram polyominoes will be utilized in the following dissection of an $(n,b)$-domino tower.

In contrast to the standard wasp-waist technique that dissects a polyomino at a thin place, this dissection into two domino towers will take place at the widest row of the tower, so that the tower is divided in such a way that rows of maximum length and all rows at a level higher than these are contained in one tower, while all rows of level lower than that of the lowest level of a row of maximum length are in another tower.  (Figure~\ref{break_fig} illustrates this process.)  The lower tower will be called a \textit{supporting} domino tower, and may be empty, while the upper tower will be called a \textit{domino stack} or a \textit{right-skewed} or \textit{left-skewed} domino tower.  We will find generating functions for each of these classes of towers and combine these results to count all convex domino towers.  
\begin{figure}[h]
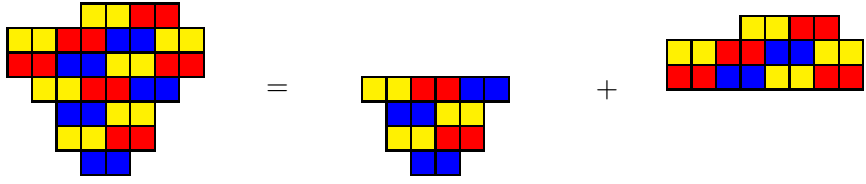

\centering
  \ytableausetup{smalltableaux}
\begin{minipage}{.2\textwidth}
  \begin{ytableau}
   \none &   \none   & \none &*(yellow) & *(yellow) &*(red) & *(red) \\
 *(yellow) & *(yellow)&*(red) & *(red) &*(blue) & *(blue) &*(yellow) & *(yellow)\\
   *(red) & *(red)   &*(blue) & *(blue)&*(yellow) & *(yellow)&*(red) & *(red)\\
   \none &  *(yellow) & *(yellow) &*(red) & *(red)  &*(blue) & *(blue)\\
   \none &   \none & *(blue) & *(blue)  &*(yellow) & *(yellow)\\
   \none &  \none &  *(yellow) & *(yellow)&*(red) & *(red)\\
   \none &  \none & \none & *(blue) & *(blue) 
  \end{ytableau}
  \end{minipage}
\begin{minipage}{.05\textwidth}
=
  \end{minipage}
\begin{minipage}{.2\textwidth}
\begin{ytableau}
\none\\
\none\\
\none\\
   \none &  *(yellow) & *(yellow) &*(red) & *(red)  &*(blue) & *(blue)\\
   \none &   \none & *(blue) & *(blue)  &*(yellow) & *(yellow)\\
   \none &  \none &  *(yellow) & *(yellow)&*(red) & *(red)\\
   \none &  \none & \none & *(blue) & *(blue) 
  \end{ytableau}
    \end{minipage}
\begin{minipage}{.05\textwidth}
+
  \end{minipage}
\begin{minipage}{.2\textwidth}
  \begin{ytableau}
     \none &   \none   & \none &*(yellow) & *(yellow) &*(red) & *(red) \\
 *(yellow) & *(yellow)&*(red) & *(red) &*(blue) & *(blue) &*(yellow) & *(yellow)\\
   *(red) & *(red)   &*(blue) & *(blue)&*(yellow) & *(yellow)&*(red) & *(red)\\
     \none\\
     \none\\
     \none
  \end{ytableau}
    \end{minipage}
\caption{Dividing a convex $(18,4)$-domino tower}
\label{break_fig}
\end{figure}

We begin by enumerating supporting domino towers created from the lower tower of a dissection.

\begin{definition}
Given a convex domino tower, identify the row of maximum length, $b$, which is at minimum level $\ell \geq 0$.  For some $n\geq 0$, the tower of $n$ blocks consisting of all dominoes on levels strictly less than $\ell$ is called a \textit{supporting $(n,b)$-domino tower}. 
\end{definition}
Figure~\ref{supporting_fig} illustrates supporting $(8,4)$-domino towers.  We note, in supporting domino towers $b$ refers to the length the base of another domino tower to be placed upon the supporting tower.  Therefore, the highest level of the supporting domino tower has length $b-1$.  Though this notion may seem counterintuitive, it will be helpful when adjoining the upper and lower domino towers.  The following proposition gives a recurrence on supporting domino towers.

\begin{figure}[h]
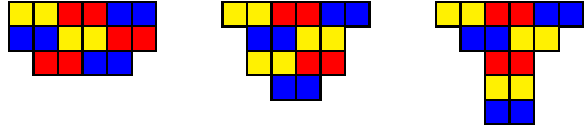

\centering
  \ytableausetup{smalltableaux}
  \begin{ytableau}
*(yellow) & *(yellow) &*(red) & *(red)  &*(blue) & *(blue)\\
 *(blue) & *(blue)&*(yellow) & *(yellow) &*(red) & *(red) \\
   \none & *(red) & *(red)  &*(blue) & *(blue)
  \end{ytableau}
\qquad
  \begin{ytableau}
  *(yellow) & *(yellow) &*(red) & *(red)  &*(blue) & *(blue)\\
   \none & *(blue) & *(blue)  &*(yellow) & *(yellow)\\
  \none &  *(yellow) & *(yellow)&*(red) & *(red)\\
  \none & \none & *(blue) & *(blue) 
  \end{ytableau}
  \qquad
    \begin{ytableau}
  *(yellow) & *(yellow) &*(red) & *(red)  &*(blue) & *(blue)\\
   \none & *(blue) & *(blue)  &*(yellow) & *(yellow)\\
  \none &  \none & *(red) & *(red)\\
  \none & \none & *(yellow) & *(yellow)\\
  \none & \none & *(blue) & *(blue) 
  \end{ytableau}
\caption{Examples of supporting $(8,4)$-domino towers}
\label{supporting_fig}
\end{figure}

\begin{proposition}\label{supporting_recurrence_prop}
Let $g(n,b)$ be the number of supporting (n,b)-domino towers.  Then
\begin{equation}
g_b(n) = g_b(n-b+1)+g_{b-1}(n-b+1)
\end{equation}
where $g_b(b-1)=1$ for $b\geq2$ and $g_b(n)=0$ for $n<1$, $b<2$, and $n<b-1$.
\end{proposition}

\begin{proof}
Because a supporting $(n,b)$-domino tower is broken from a larger convex domino tower with a row a length $b$, the supporting tower must consist of a base of length $j$ for some $j\geq 1$ and rows of every length $i$ where $j\leq i \leq b-1$. (If a row of length $i$ is missing the tower could not be built because a row of length $i-1$ cannot hold a row of length $i+1$ or greater.)  Further the supporting tower must have a row of length $(b-1)$ at the highest level in order to support the row of length $b$ is the upper tower.  Thus, a supporting domino tower can be built recursively by adding a top row of $(b-1)$ blocks to a domino tower of $(n-b+1)$ blocks that has either top row of length $(b-1)$ or of length $(b-2)$.  These towers are counted by the functions $g_b(n-b+1)$ and $g_{b-1}(n-b+1)$.  Easily, the initial conditions of the recurrence are satisfied where $g_b(b-1)=1$ corresponds to the domino tower of height one consisting of a single horizontal row of $(b-1)$ dominoes and the result is shown.
\end{proof}

The recurrence in Proposition~\ref{supporting_recurrence_prop} can be used to find the generating function for the number of supporting domino towers of $n$ dominoes which can support a tower with base $b$.

\begin{proposition}\label{prop_supporting_genfunc}
For $b \geq 1$, the generating function for the number of supporting $(n,b)$-domino towers is
\begin{equation*}
G_b(x) = \sum_{n\geq 0}g_b(n) x^n = \sum_{i=1}^{b-1} \prod_{j=1}^{i} \frac{x^{b-j}}{1-x^{b-j}}.
\end{equation*}
\end{proposition}

\begin{proof}
We proceed by induction.  In the initial cases $g_1(n) =0$ for all $n$ corresponding to the empty sum, and $g_2(n)=1$, a column of $n$ dominoes, so
\begin{equation*}G_2(x) = \sum_{n\geq 0} g_2(n)x^n = x+x^2+x^3+x^4+\cdots = \frac{x}{1-x} = \sum_{i=1}^{2-1} \prod_{j=1}^i \frac{x^{2-j}}{1-x^{2-j}}.
\end{equation*}
Applying the recurrence in Proposition~\ref{supporting_recurrence_prop} for the inductive step, we have
\begin{eqnarray*}
G_b(x) &=& \sum_{n\geq 0} g_b(n)x^n = \sum_{n\geq 0} g_b(n-b+1)x^n+\sum_{n\geq 0} g_{b-1}(n-b+1)x^n\\
&=&x^{b-1} \sum_{n\geq 0}g_b(n-b+1)x^{n-b+1} +x^{b-1}\sum_{n\geq 0} g_{b-1}(n-b+1)x^{n-b+1}\\
&=&x^{b-1} G_b(x) +x^{b-1}G_{b-1}(x) + x^{b-1}\\
\end{eqnarray*}
with the $x^{b-1}$ term in the sum due to the initial condition $g_b(b-1) = 1$.  Thus,
\begin{eqnarray*}
G_b(x) &=& \frac{x^{b-1}}{1-x^{b-1}}\left(G_{b-1}(x) +1\right)\\
&=& \frac{x^{b-1}}{1-x^{b-1}}\left(\sum_{i=1}^{b-2} \prod_{j=1}^{i} \frac{x^{b-1-j}}{1-x^{b-1-j}}+1\right)\\
&=& \sum_{i=1}^{b-1} \prod_{j=2}^{i} \frac{x^{b-j}}{1-x^{b-j}} +\frac{x^{b-1}}{1-x^{b-1}}\\
&=& \sum_{i=1}^{b-1} \prod_{j=1}^{i} \frac{x^{b-j}}{1-x^{b-j}}
\end{eqnarray*}
and we have proven the claim.
\end{proof}

We note that the total number of supporting domino towers with $n$ dominoes is a given by sequence A034296~\cite{OEIS} counting the number of flat partitions of $n$.  The values $g_b(n)$, can consequently be viewed as the number of flat partitions of $n$ whose largest part is $b-1$ and are found in sequence A117468~\cite{OEIS}.  (Another version of the generating function using an infinite sum is also given in the entry.)

Having enumerated the lower towers of the dissection, we now consider the upper tower, beginning with domino stacks, which are a subset of \textit{stacks} described by Wright~\cite{Wright}. Wright found generating functions for general polyomino stacks in area partitioned by number of rows as well as in parameters area, the length of the top, and the length of the base partitioned by the number of rows.
\begin{definition}
A \textit{$(n,b)$-domino stack} is a convex domino tower of $n$ dominoes with base $b$ dominoes such that all columns of the polyomino intersect the base.  
\end{definition}
See Figure~\ref{concave_fig} for examples of some domino stacks and Table~\ref{concave_table} or sequence A275204~\cite{OEIS} for the number of $(n,b)$-domino stacks for $1\leq n\leq 10$.  As before, we begin with a recurrence.

\begin{figure}[h]
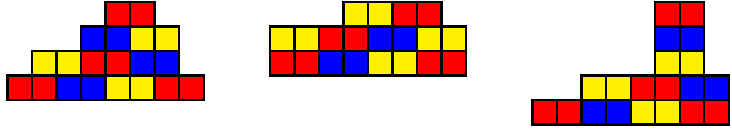

\centering
  \ytableausetup{smalltableaux}
  \begin{ytableau}
   \none & \none & \none & \none &*(red) & *(red) \\
   \none &   \none   & \none &*(blue) & *(blue) &*(yellow) & *(yellow) \\
 \none &*(yellow) & *(yellow)&*(red) & *(red) &*(blue) & *(blue) \\
   *(red) & *(red)   &*(blue) & *(blue)&*(yellow) & *(yellow)&*(red) & *(red)
  \end{ytableau}
\qquad
  \begin{ytableau}
   \none &   \none   & \none &*(yellow) & *(yellow) &*(red) & *(red) \\
 *(yellow) & *(yellow)&*(red) & *(red) &*(blue) & *(blue) &*(yellow) & *(yellow)\\
   *(red) & *(red)   &*(blue) & *(blue)&*(yellow) & *(yellow)&*(red) & *(red)
  \end{ytableau}
  \qquad
    \begin{ytableau}
  \none & \none & \none & \none & \none &*(red) & *(red) \\
       \none &       \none & \none & \none & \none &*(blue) & *(blue) \\
   \none&    \none & \none & \none & \none &*(yellow) & *(yellow) \\
   \none &   \none   &*(yellow) & *(yellow)&*(red)&*(red) & *(blue) & *(blue) \\
   *(red) & *(red)   &*(blue) & *(blue)&*(yellow) & *(yellow)&*(red) & *(red)
  \end{ytableau}
\caption{Examples of $(10,4)$-domino stacks}
\label{concave_fig}
\end{figure}

\begin{proposition}\label{concave_recurrence_prop}
Let $h_b(n)$ denote the number of $(n,b)$-domino stacks.  Then
\begin{equation*}
h_b(n) = \sum_{i=1}^b (2(b-i)+1)h_i(n-b)
\end{equation*}
where $h_b(b)=1$ and $h_b(n)=0$ if $n,b<1$ or $n<b$.
\end{proposition}

\begin{proof}
We may build a domino stack with base $b$ by placing any domino stack with base $i$ where $1\leq i \leq b$ on the base.  Place the $(n-b,i)$-domino stack in the left-most position, aligning with left edge of the new base in one way.  From here there are $b-i$ dominoes or $2(b-i)$ unit squares to the right.  The $(n-b,i)$-domino stack may also be placed in any of those positions, giving $2(b-i)+1$ positions for each $(n-b,i)$-domino stack.  The initial case, $h_b(b)$, is the single horizontal polyomino with a base of $b$ dominoes, and we have shown the recurrence.
\end{proof}

Now, the generating function can be given for $(n,b)$-domino stacks.
\begin{proposition}
For $b \geq 1$,  the generating function for the number of $(n,b)$-domino stacks is
\begin{equation*}
H_b(x) = \sum_{n\geq 0} h_b(n) x^n = \frac{x^b}{1-x^b} \sum_S \prod_{j=1}^{m} \frac{(2(k_{j+1}-k_j)+1)x^{k_j}}{1-x^{k_j}}
\end{equation*}
where the sum is over all subsets $S \subseteq \{1, 2, \cdots, b-1\}$ such that the elements of $S$ are ordered as $S = \{k_1 < k_2 < \cdots <k_m\}$ and $k_{m+1}=b$.
\end{proposition}

\begin{proof}
We proceed by induction.  In the initial case, by counting columns of dominoes with base of length 1, we have
\begin{equation*}H_1(x) = \sum_{n\geq 0} h_1(n)x^n = x+x^2+x^3+x^4+\cdots = \frac{x}{1-x}.
\end{equation*}
In the formula, as the only subset of the empty set is itself, we have a sum over the empty set giving an empty product and thus a sum of 1.  This is multiplied by $\frac{x}{1-x}$ and we have proven the initial case.

Next, utilizing the recurrence in Proposition~\ref{concave_recurrence_prop}, we have
\begin{eqnarray*}
H_b(x) &=& \sum_{n\geq 0} h_b(n)x^n = \sum_{n\geq 0} \sum_{i=1}^b  (2(b-i)+1)h_{i}(n-b)x^n\\
&=&x^{b} \sum_{n\geq 0}(2b-1)h_1(n-b)x^{n-b} +x^{b}\sum_{n\geq 0} (2b-3)h_{2}(n-b)x^{n-b} + \cdots + x^b\sum_{n\geq 0} h_b(n-b)\\
&=&(2b-1)x^{b} H_1(x) +(2b-3)x^{b}H_{2}(x) + \cdots + x^bH_b(x)+x^{b}\\
\end{eqnarray*}
where $x^{b}$ is given by the initial condition $h_b(b) = 1$.  Thus,
\begin{eqnarray*}
H_b(x) &=& \frac{x^{b}}{1-x^{b}}\left(\sum_{i=1}^{b-1}(2(b-i)+1)H_{i}(x) +1\right)\\
&=& \frac{x^{b}}{1-x^{b}}\left(\left( \sum_{i=1}^{b-1} \frac{(2(b-i)+1) x^i}{1-x^i}\sum_{\stackrel{S \subseteq \{1, 2, \cdots, i-1\} s.t.}{{S = \{k_1 < k_2 < \cdots <k_m\} \textit{and } k_{m+1}=i}}} \prod_{j=1}^{m} \frac{(2(k_{j+1}-k_j)+1)x^{k_j}}{1-x^{k_j}}\right) + 1\right)\\
&=& \frac{x^{b}}{1-x^{b}}\left(\sum_{\stackrel{S \subseteq \{1, 2, \cdots, b-1\}, S\not=\emptyset s.t.}{{S = \{k_1 < k_2 < \cdots <k_m\} \textit{and } k_{m+1}=b}}} \prod_{j=1}^{m} \frac{(2(k_{j+1}-k_j)+1)x^{k_j}}{1-x^{k_j}} + 1\right)\\
&=& \frac{x^{b}}{1-x^{b}}\sum_{\stackrel{S \subseteq \{1, 2, \cdots, b-1\} s.t.}{{S = \{k_1 < k_2 < \cdots <k_m\} \textit{and } k_{m+1}=b}}} \prod_{j=1}^{m} \frac{(2(k_{j+1}-k_j)+1)x^{k_j}}{1-x^{k_j}} \\
\end{eqnarray*}
and we have proven the claim.
\end{proof}

\begin{table}
\[
\begin{array}{l|rrrrrrrrrr|c}
n\backslash b& 1&2&3&4&5&6&7&8&9&10&\hbox{Total}\\
\hline
1 & 1 &0 &0&0&0&0&0&0&0&0&1\\
2&1&1&0&0&0&0&0&0&0&0&2\\
3&1&3&1&0&0&0&0&0&0&0&5\\
4&1&4&5&1&0&0&0&0&0&0&11\\
5&1&6&8&7&1&0&0&0&0&0&23\\
6&1&7&15&12&9&1&0&0&0&0&45\\
7&1&9&22&25&16&11&1&0&0&0&85\\
8&1&10&31&43&35&20&13&1&0&0&154\\
9&1&12&41&68&65&45&24&15&1&0&267\\
10&1&13&54&99&113&87&55&28&17&1&455\\
\end{array}
\]
\caption{Table enumerating $(n,b)$-domino stacks}
\label{concave_table}
\end{table}

Finally, we consider skewed $(n,b)$-domino towers.  These are subset of parallelogram polyominoes of area $2n$ which exclude rectangular polyominoes.
\begin{definition}
A \textit{right-skewed $(n,b)$-domino tower} is parallelogram $(n,b)$-domino tower such that there exists at least one column of the polyomino to the right of the base that does not intersect the base of the polyomino.  Similarly, define a \textit{left-skewed $(n,b)$-domino tower} to be a reflection of a right-skewed domino tower across a vertical axis.  
\end{definition}
Figure~\ref{skew_fig} illustrates some right- and left-skewed domino towers.  Of course, it is immediate that the number of right-skewed $(n,b)$-domino towers is equal to the number of left-skewed $(n,b)$-domino towers.  The following proposition gives the recurrence on right- or left-skewed $(n,b)$-domino towers.

\begin{figure}[h]
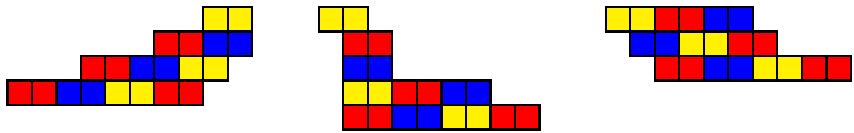

\centering
  \ytableausetup{smalltableaux}
  \begin{ytableau}
\none&\none&  \none & \none& \none & \none & \none & \none &*(yellow) & *(yellow) \\
\none&\none& \none&  \none &   \none   & \none&*(red) & *(red)  &*(blue) & *(blue)  \\
\none& \none & \none &*(red) & *(red) &*(blue) & *(blue)&*(yellow) & *(yellow) \\
   *(red) & *(red)   &*(blue) & *(blue)&*(yellow) & *(yellow)&*(red) & *(red)
  \end{ytableau}
  \qquad
    \begin{ytableau}
*(yellow) & *(yellow) \\
\none &*(red) & *(red) \\
\none &*(blue) & *(blue) \\
 \none &*(yellow) & *(yellow)&*(red)&*(red) & *(blue) & *(blue) \\
 \none &  *(red) & *(red)   &*(blue) & *(blue)&*(yellow) & *(yellow)&*(red) & *(red)
  \end{ytableau}
\qquad
  \begin{ytableau}
*(yellow) & *(yellow) &*(red) & *(red) &*(blue) & *(blue)\\
\none&*(blue) & *(blue) & *(yellow) & *(yellow)&*(red) & *(red) \\
  \none &\none   & *(red) & *(red)   &*(blue) & *(blue)&*(yellow) & *(yellow)&*(red) & *(red)
  \end{ytableau}
\caption{Examples of right-skewed or left-skewed $(10,4)$-domino towers}
\label{skew_fig}
\end{figure}

\begin{proposition}\label{skew_recurrence_prop}
Let $r_b(n)$ denote the number of right-skewed $(n,b)$-domino towers.  Then
\begin{equation}
r_b(n) = \sum_{i=1}^b \left(2 r_i(n-b)+h_i(n-b) \right)
\end{equation}
where $r_b(b+1)=1$ and $r_b(n)=0$ if $n<2$, $b<1$, or $n<b+1$.
\end{proposition}

\begin{proof}
Given a base of $b$ dominoes, a right-skewed $(n,b)$-domino tower can be built by placing a right-skewed $(n-b,i)$-domino tower, where $i$ ranges from $1$ to $b$, on the base in two ways:  either the base of the $(n-b,i)$-domino tower is right-aligned with the new base of length $b$ or it hangs over by one unit square.  Further, a right-skewed tower could also be created by placing a $(n-b,i)$-domino stack on the new base so that it also overhangs by one unit square for any $1\leq i \leq b$.  The initial condition describes a horizontal base of size $b$ with a single domino over the rightmost block of the base.
\end{proof}

The generating function for right- or left-skewed domino towers follows in Proposition~\ref{gen_function_skew}.

\begin{proposition}
For $b\geq 1$,  the generating function the number of right-skewed $(n,b)$-domino towers is
\begin{equation*}
R_b(x) = \sum_{n\geq 0} r_b(n) x^n= \frac{x^b}{1-2x^b} \sum_{j=1}^{b} H_i(x) \left( \sum_{S\subseteq \{j,j+1, \ldots, b-1\}} \prod_{k\in S} \frac{2x^k}{1-2x^k} \right).
\end{equation*}
\label{gen_function_skew}
\end{proposition}

\begin{proof}
Our initial case is $R_1(x)$.  Consider $r_1(n)$ a tower of $n$ dominoes and base $1$.  Each of the $n-1$ non-base dominoes must be placed on top of the domino previously placed domino to preserve the parallelogram shape.  Further it may be placed of the previous domino by right-alignment or a shift to the right by one unit square.  This gives $2^{n-1}$ possibilities of which we omit the rectangle tower of $n$ blocks already counted as domino stacks, so $r_1(n) = 2^{n-1}-1$ for $n\geq 1$.  Thus,
\begin{equation*}
R_1(x) = \sum_{n\geq 1 } (2^{n-1}-1)x^n = \sum_{n\geq 1 } (2x)^n - \sum_{n\geq 1} x^n = \frac{x}{1-2x} - \frac{x}{1-x} = \frac{x^2}{(1-x)(1-2x)}
\end{equation*}
which is also given by the formula,
\begin{equation*}
R_1(x) = \frac{x}{1-2x} H_1(x) \sum_{S=\emptyset} 1= \frac{x^2}{(1-x)(1-2x)}.
\end{equation*}
We induct on $b\geq 1$ and apply Proposition~\ref{skew_recurrence_prop}.
\begin{eqnarray*}
R_b(x) &=& \sum_{n\geq 0} r_b(n) x^n = \sum_{n \geq 0} \sum_{i=1}^b \left(2 r_i(n-b)+h_i(n-b) \right)\\
&=& \sum_{n \geq 0}2 r_1(n-b)x^n+\sum_{n\geq 0}h_1(n-b)x^n+ \cdots +\sum_{n \geq 0} 2 r_b(n-b)x^n+\sum_{n\geq0}h_b(n-b)x^n\\
&=&2x^b R_1(x) + 2x^b R_2(x) + \cdots + 2x^b R_b(x) + x^bH_1(x) + x^bH_2(x) + \cdots + x^bH_b(x)\\
\end{eqnarray*}
Therefore,
\begin{eqnarray*}
R_b(x)&=& \frac{x^b}{1-2x^b}\left(\sum_{i=1}^{b-1} \left(2R_i(x)+H_i(x)\right) + H_b(x)\right)\\
&=& \frac{x^b}{1-2x^b}\left( \sum_{i=1}^{b-1}  \left( \frac{2x^i}{1-2x^i} \sum_{j=1}^{i} H_j(x) \left( \sum_{S\subseteq \{j,j+1, \ldots, i-1\}} \prod_{k\in S} \frac{2x^k}{1-2x^k} \right) + H_i(x) \right) +H_b(x)\right)\\
&=&\frac{x^b}{1-2x^b} \sum_{j=1}^{b-1} H_i(x) \left( \sum_{S\subseteq \{j,j+1, \ldots, b-1\}} \prod_{k\in S} \frac{2x^k}{1-2x^k} \right) +H_b(x)\\
&=&\frac{x^b}{1-2x^b} \sum_{j=1}^{b} H_i(x) \left( \sum_{S\subseteq \{j,j+1, \ldots, b-1\}} \prod_{k\in S} \frac{2x^k}{1-2x^k} \right)
\end{eqnarray*}
and we have proven the result.
\end{proof}

Table~\ref{skewed_table} counts the number of right- or left-skewed domino towers for $1\leq n\leq 10$ which can also be found in sequence A275599~\cite{OEIS}.
\begin{table}
\[
\begin{array}{l|rrrrrrrrr|c}
n\backslash b& 1&2&3&4&5&6&7&8&9&\hbox{Total}\\
\hline
1 & 0 &0 &0&0&0&0&0&0&0&0\\
2&1&0&0&0&0&0&0&0&0&1\\
3&3&1&0&0&0&0&0&0&0&4\\
4&7&4&1&0&0&0&0&0&0&12\\
5&15&12&4&1&0&0&0&0&0&32\\
6&31&27&13&4&1&0&0&0&0&76\\
7&63&61&34&13&4&1&0&0&0&176\\
8&127&124&77&35&13&4&1&0&0&381\\
9&255&258&165&86&35&13&4&1&0&817\\
10&511&513&348&185&87&35&13&4&1&1697\\
\end{array}
\]
\caption{Table enumerating right-skewed $(n,b)$-domino towers}
\label{skewed_table}
\end{table}

We are now ready to enumerate convex $(n,b)$-domino towers, that is, convex domino towers consisting of $n$ dominoes with widest row having $b$ dominoes.  The generating function is found in Theorem~\ref{convex_theorem} and counts for the number of convex $(n,b)$-domino towers for $1\leq n \leq 10$ are shown in Table~\ref{convex_table} or in sequence A275662~\cite{OEIS}.

\begin{theorem}\label{convex_theorem}
For $b\geq 1$, the generating function for the number of convex $(n,b)$-domino towers consisting of $n$ dominoes with the row of maximum length having $b$ dominoes is
\begin{equation*}
C_b(x) = \sum_{n\geq 0} c_b(n)x^n = (G_b(x) +1) (2R_b(x)+H_b(x)).
\end{equation*}
\end{theorem}

\begin{proof}
To build a convex $(n,b)$-domino tower we need a supporting tower, which is found $G_b(x)$ ways for $b\geq 1$ along with the possibility of an empty supporting tower 1 way, as well as either a right-skewed or left-skewed domino tower or a domino stack enumerated by $2R_{b}(x)+H_b(x)$ for $b\geq 1$.
\end{proof}

We make the following observation on extensions of these results to $k$-omino towers for $k>2$.

\begin{table}
\[
\begin{array}{l|rrrrrrrrrr|r}
n\backslash b& 1&2&3&4&5&6&7&8&9&10&\hbox{Total}\\
\hline
1&1&0&0&0&0&0&0&0&0&0&1\\
2&3&1&0&0&0&0&0&0&0&0&4\\
3&7&6&1&0&0&0&0&0&0&0&14\\
4&15&18&7&1&0&0&0&0&0&0&41\\
5&31&48&17&9&1&0&0&0&0&0&106\\
6&63&109&49&20&11&1&0&0&0&0&253\\
7&127&240&115&52&24&13&1&0&0&0&541\\
8&255&498&258&122&61&28&15&1&0&0&1234\\
9&511&1026&551&261&136&71&32&17&1&0&2598\\
10&1023&2065&1163&531&298&157&81&36&19&1&5340
\end{array}
\]
\caption{Table of values for $c_b(n)$, the number of convex $(n,b)$-domino towers}
\label{convex_table}
\end{table}

\begin{remark}
The recurrences in Propositions~\ref{supporting_recurrence_prop}, \ref{concave_recurrence_prop}, and \ref{skew_recurrence_prop}, respectively, have generalizations to convex $k$-omino towers, where
\begin{eqnarray*}
g_{b,k}(n)&=&g_{b,k}(n-b+1)+ (k-1)g_{b-1,k}(n-b+1),\\
h_{b,k}(n)&=&\sum_{i=1}^b (k(b-i)+1)h_{i,k}(n-b), \hbox{ and} \\
r_{b,k}(n)&=&\sum_{i=1}^b (k r_{i,k}(n-b) + (k-1)h_{i,k}(n-b),\\
\end{eqnarray*}
respectively.  Thus, results on generating functions for $k>2$ may be obtained in the general case in a similar fashion, but are omitted here.
\end{remark}

In the final section, we describe the asymptotic behavior of generating function $C_b(x)$ to give an approximation of the coefficients, that is, the number of convex domino towers on $n$ dominoes.
%%%%%%%%%%%%%%%%%%%%%%%%%%%%%%%%%%%%%%%%%%%%%%%%%%%%%%%%%%%%%%%%%

\section{Asymptotic behavior of convex domino towers}\label{asymptotic_behavior}
As a product of linearly recursive functions, the function $C_b(z)$ is a rational generating function of the form $C_b(z)=f(z)/g(z)$ for some polynomials $f$ and $g$, and as such, when viewed as an analytic function in the complex plane, its asymptotic approximations can be described in terms of poles with smallest modulus.  In particular, we have the following proposition.

\begin{proposition}
The coefficients of $C_{b}(z)$, the generating function on the number of convex $(n,b)$-domino towers, are of exponential order $2^n$, that is,
\begin{equation*}
[z]^n C_b(z) \sim \theta_b(n) 2^n
\end{equation*}
where $\theta_b(n)$ is constant for each $b>1$.
\end{proposition}

\begin{proof}
The function $C_b(z)$ is given by a product of functions.  The first factor is a sum of the functions $R_b(z)$ and $H_b(z)$.  Summands in the function $R_b(z)$ have denominators that are products on $(1-2z^i)$ and summands in the function $H_b(z)$ have denominators which are products on $(1-z^i)$ where $i$ takes on values from $1$ to $b$.  The function $G_b(z)$ also has a denominator given by products of $(1-z^i)$, in this case, where $1\leq i \leq b-1$  Thus, the denominator of $C_b(z)$ is given by 
\begin{equation*}
g_b(z) = (1-z^b)(1-2z^b)\prod_{k=1}^{b-1} (1-z^k)^2(1-2z^k),
\end{equation*}
and therefore the unique pole of smallest modulus is $1/2$ with multiplicity $1$.  
\end{proof}

Next, we wish to describe the sub-exponential factor $\theta_b(n)$.  
\begin{theorem}\label{subex_theorem}
If $C_b(x)$ is the generating function for the number of convex domino towers, the sub-exponential factor $\theta_b(n)$ in the approximation $[z]^n C_b(z) \sim \theta_b(n) 2^n$ approaches 
\begin{equation*}
\theta_b(n) \sim (1/2)^{b-1} \prod_{k=1}^{b-1} \frac{2^k}{2^k-1}
\end{equation*}
as $b>>0$, and consequently may be approximated by $3.46 (1/2)^{b-1}$.
\end{theorem}

In order prove Theorem~\ref{subex_theorem}, we use the following result from analytic combinatorial analysis, see Flajolet and Sedgewick~\cite{Flajolet_Sedgewick}, for example, that given exponential order $2^n$ the sub-exponential factor is as follows:
\begin{equation*}
\theta_b(n) = \frac{(-2) f(1/2)}{g^{(1)}(1/2)}.
\end{equation*}

First, we have the following proposition.

\begin{proposition}\label{derivative_half_prop}
The derivative $g_b'(1/2)$ is
\begin{equation*}
g_b'(1/2) = (-2)\left(\frac{2^b-1}{2^b}\right) \left( \prod_{k=1}^{b-1} \left(\frac{2^k -1}{2^k}\right)^3 \right)
\end{equation*}
\end{proposition}
\begin{proof}
We apply the product rule to the factors $(1-z^b)(1-z)^2\prod_{k=2}^{b-1} (1-z^k)^2(1-2z^k)$ and $(1-2z)$ before substituting $z=1/2$.
\begin{eqnarray*}
g_b'(z) &=& (1-2z) \frac{d}{dz} \left((1-z^b)(1-2z^b)(1-z)^2\prod_{k=2}^{b-1} (1-z^k)^2(1-2z^k)\right) + \\
&&\left(\frac{d}{dz} (1-2x)\right) \left((1-z^b)(1-2z^b)(1-z)^2\prod_{k=2}^{b-1} (1-z^k)^2(1-2z^k)\right)\\
g_b'(1/2) &=& 0 +(-2)(1-(1/2)^b) (1-(1/2)^{b-1})(1-(1/2))^2\prod_{k=2}^{b-1} (1-(1/2)^k)^2(1-(1/2)^{k-1}) \\
&=& (-2)\left(\frac{2^b-1}{2^b}\right) \left(\frac{2^{b-1}-1}{2^{b-1}}\right)\left(\frac{1}{2}\right)^2 \prod_{k=2}^{b-1}\left(\frac{2^k-1}{2^k}\right)^2\left(\frac{2^{k-1}-1}{2^{k-1}}\right)\\
&=& (-2)\left(\frac{2^b-1}{2^b}\right) \left( \prod_{k=1}^{b-1} \left(\frac{2^k -1}{2^k}\right)^3 \right)
\end{eqnarray*}
\end{proof}

The numerator $f_b(z)$ of $C_b(z)$ is more complicated to work with.  Let us first try to determine the numerator of rational function given by the first factor of $C_b$, that is, the numerator of the rational function $2R_b(z) + H_b(z)$.

\begin{proposition}\label{numerator_left_prop}
The numerator $\hat{f}(1/2)$ of the rational function $\hat{f}(z)/\hat{g}(z)=2R_b(z) + H_b(z)$ is given by 
\begin{equation*}
\hat{f}(1/2) =(1/2)^b\cdot \prod_{k=2}^b \frac{2^k-1}{2^k}
\end{equation*}
for $b\geq 1$.
\end{proposition}
\begin{proof}
Again we shall simplify the calculations by looking for summands in $\hat{f}(z)$ which contain the factor $(1-2z)$ as these will go to zero when $z=1/2$ is substituted.  First consider the function $2R_b(z) + H_b(z)$ as given below.
\begin{equation*}
\frac{2z^b}{1-2z^b} \sum_{j=1}^b H_j(z) \left(\sum_{S\subseteq \{j,j+1, \ldots, b-1\}} \prod_{k\in S} \frac{2z^k}{1-2z^k} \right) + H_b(z).
\end{equation*}
One can see that the common denominator, $\hat{g}(z)$ is of the form $\prod_{i=1}^b (1-z^i)(1-2z^i)$.  In order to simplify the calculation of the denominator $\hat{f}(1/2)$, we note the denominators of the summands in $H_b(z)$ do not contain the factor $(1-2z)$, so when finding a common denominator these corresponding summands in the numerator must contain $(1-2z)$ as a factor and hence go to zero when substituting $z=1/2$.  Further, in the sum of the subsets $S$ of the set $\{j,j+1, \ldots, b-1\}$, only the case where $j=1$ will have factors of $(1-2z)$ in the denominator and hence not in the corresponding factor in the function $\hat{f}(z)$.  Thus we need only consider the contribution of the following terms to the function $\hat{f}(z)$.
\begin{equation*}
\frac{2z^b}{1-2z^b} H_1(z) \left( \sum_{\stackrel{S \subseteq \{1, 2, \ldots, b-1\}}{s.t. 1\in S}}\prod_{k\in S} \frac{2z^k}{1-2z^k} \right) = \left(\frac{2z^b}{1-2z^b}\right) \left( \frac{z}{1-z} \right) \left(\frac{2z}{1-2z}\right) \left( \sum_{S \subseteq \{2, 3, \ldots, b-1\}}\prod_{k\in S} \frac{2z^{k}}{1-2z^{k}} \right)
\end{equation*}
Thus after finding the common denominator and summing, the contribution to $\hat{f}(z)$ is given by
\begin{equation*}
2z^b \cdot z \cdot 2z \cdot \prod_{k=2}^b (1-z^k) \left(\sum_{S \subseteq \{2, \ldots, b-1\}}\prod_{k\in S} 2z^{k} \cdot \prod_{k\notin S} (1-2z^k) \right).
\end{equation*}
Substituting in $z=1/2$, we have
\begin{eqnarray*}
&&2(1/2)^b \cdot (1/2) \cdot 2(1/2) \cdot \prod_{k=2}^b (1-(1/2)^k) \sum_{S \subseteq \{2, \ldots, b-1\}}\prod_{k\in S} (1/2)^{k-1} \cdot \prod_{k\notin S} (1-(1/2)^{k-1})\\
&=& (1/2)^b \cdot \prod_{k=2}^b \frac{2^k-1}{2^k} \sum_{S \subseteq \{1, \ldots, b-2\}} \prod_{k\in S} (1/2)^k \cdot \prod_{k\notin S} \frac{2^k-1}{2^k}\\
&=& (1/2)^b \cdot \prod_{k=2}^b \frac{2^k-1}{2^k}  \sum_{S \subseteq \{1, \ldots, b-2\}} \frac{\prod_{k\notin S} 2^k-1}{\prod_{k=1}^{b-2} 2^k}\\
&=& (1/2)^b \cdot \prod_{k=2}^b \frac{2^k-1}{2^k} \left(\frac{ \sum_{S \subseteq \{1, \ldots, b-2\}}\prod_{k\in S} 2^k-1}{2^{b-1\choose 2}}\right)\\
&=& (1/2)^b \cdot \prod_{k=2}^b \frac{2^k-1}{2^k} 
\end{eqnarray*}
The penultimate equality holds because summing over all subsets by elements not in a subset is equivalent to summing over all subsets with elements in the subset.  The last equality holds because a set, $A$, of cardinality $b-1\choose 2$ may be partitioned into disjoint sets $A_1, A_2, \ldots, A_{b-2}$ such that the cardinality of $A_k$ is $k$.  Subsets of $A$ may be counted directly by $2^{b-1\choose 2}$ or by taking the union of subsets, one each from the sets $A_1, A_2, \ldots, A_{b-2}$ where $k\in S$ indicates $2^k-1$ possible nonempty subsets to be contributed to the union from the set $A_k$, and $k\notin S$ indicates the empty set is contributed to the union from the set $A_k$.
\end{proof}

Now we are ready to state our theorem on the asymptotic behavior of $C_b(z)$.

\begin{theorem}\label{asymptotic_theorem}
The number of convex $(n,b)$-domino towers, given by the $n^{th}$ coefficient of $C_b(z)$, has an asymptotic approximation as follows:
\begin{equation*}
[z]^n C_b(z) \sim 2^{n-b+1} \prod_{k=1}^{b-1} \frac{2^k}{2^k-1} \cdot \left( 1 + \sum_{i=0}^{b-2} \frac{1}{ \prod_{k=i+1}^{b-1} (2^k -1)}\right)
\end{equation*}
\end{theorem}

\begin{proof}
We will now need to combine our result of Proposition~\ref{numerator_left_prop} with the factor $(1+G_b(z))$ of $C_b(z)$.  We have 
\begin{equation*}
1+G_b(z) = 1+ \frac{z^{b-1}}{1-z^{b-1}} +\left(\frac{z^{b-1}}{1-z^{b-1}}\right) \left(\frac{z^{b-2}}{1-z^{b-2}}\right)+ \cdots + \left(\frac{z^{b-1}}{1-z^{b-1}}\right) \left(\frac{z^{b-2}}{1-z^{b-2}}\right)\cdots \left(\frac{z}{1-z}\right) 
\end{equation*}
Let $\bar{f}_b(z)/\bar{g}_b(z)$ denote the rational function $(1+G_b(x))$.  The common denominator of this function is $\bar{g}_b(z)=\prod_{i=1}^{b-1} (1-z^i)$.  Thus the numerator is 
\begin{equation*}
\bar{f}_b(z)=(1-z^{b-1})(1-z^{b-2})\cdots (1-z) + z^{b-1} (1-z^{b-2})\cdots (1-z) + \cdots + (z^{b-1})(z^{b-2})\cdots (z)
\end{equation*}
and
\begin{eqnarray*}
\bar{f}_b\left(1/2\right) &=& \left(1-1/2^{b-1}\right)\left(1-1/2^{b-2}\right) \cdots \left(1-1/2\right) + \left(1/2\right)^{b-1} \left(1-1/2^{b-2}\right)\cdots \left(1-1/2\right) +\\
&& \cdots + \left(1/2^{b-1}\right)\left(1/2^{b-2}\right)\cdots (1/2)\\
&=& \left(\frac{2^{b-1}-1}{2^{b-1}}\right) \left(\frac{2^{b-2}-1}{2^{b-2}}\right) \cdots \left(\frac{1}{2}\right) +\left(\frac{1}{2}\right)^{b-1} \left(\frac{2^{b-2}-1}{2^{b-2}}\right) \cdots \left( 1-\frac{1}{2} \right)+ \cdots\\
&&  + \left(\frac{1}{2}^{b-1}\right)\left(\frac{1}{2}^{b-2}\right)\cdots \frac{1}{2}\\
&=& \frac{\left(\sum_{i=1}^{b-1} \prod_{k=1}^{i} (2^{k}-1)\right) +1}{\prod_{k=1}^{b-1} 2^k} \\
&=& \frac{\sum_{i=0}^{b-1} \prod_{k=1}^{i} (2^{k}-1)}{\prod_{k=1}^{b-1} 2^k}
\end{eqnarray*}
Thus 
\begin{equation*}
f(1/2)=\hat{f}(1/2) \cdot \bar{f}(1/2) =   \frac{(1/2)^b \cdot \prod_{k=2}^b \frac{2^k-1}{2^k}  \cdot \left(\sum_{i=0}^{b-1} \prod_{k=1}^{i} (2^{k}-1)\right)}{\prod_{k=1}^{b-1} 2^k} 
\end{equation*}
and we have 
\begin{eqnarray*}
\theta_b(n)&=&\frac{(-2) f(1/2)}{g^{(1)}(1/2)}\\
&=&\frac{(-2) (1/2)^b \cdot \prod_{k=2}^b \frac{2^k-1}{2^k}  \cdot \left(\sum_{i=0}^{b-1} \prod_{k=1}^{i} (2^{k}-1)\right)}{ (-2)\left(\frac{2^b-1}{2^b}\right)\left( \prod_{k=1}^{b-1} \left(\frac{2^k -1}{2^k}\right)^3 \right)\prod_{k=1}^{b-1} 2^k }\\
&=&\frac{ (1/2)^{b-1}  \cdot \left(\sum_{i=0}^{b-1} \prod_{k=1}^{i} (2^{k}-1)\right)} {\left( \prod_{k=1}^{b-1} \left(\frac{2^k -1}{2^k}\right)^2 \right)\prod_{k=1}^{b-1} 2^k }\\
&=&\frac{ (1/2)^{b-1}  \cdot \left(\sum_{i=0}^{b-1} \prod_{k=1}^{i} (2^{k}-1)\right)}{  \prod_{k=1}^{b-1} \frac{2^k -1}{2^k} \cdot \prod_{k=1}^{b-1} 2^k-1 }\\
&=& (1/2)^{b-1} \prod_{k=1}^{b-1} \frac{2^k}{2^k-1} \cdot \frac{\left(\sum_{i=0}^{b-1} \prod_{k=1}^{i} (2^{k}-1)\right)}{ \prod_{k=1}^{b-1} 2^k-1}\\
&=& (1/2)^{b-1} \prod_{k=1}^{b-1} \frac{2^k}{2^k-1} \cdot \left( 1 + \sum_{i=0}^{b-2} \frac{1}{ \prod_{k=i+1}^{b-1} (2^k -1)}\right)\\
\end{eqnarray*}
\end{proof}
Thus Theorem~\ref{subex_theorem} is a consequence of Theorem~\ref{asymptotic_theorem}.  The factor $\left( 1 + \sum_{i=0}^{b-2} \frac{1}{ \prod_{k=i+1}^{b-1} (2^k -1)}\right)$ approaches 1 as $b>>0$, and the factor $\prod_{k=1}^{b-1} \frac{2^k}{2^k-1}$ is monotonic sequence bounded above whose decimal expansion is given by the sequence A065446~\cite{OEIS} and can be approximated by 3.46.  Rounded and estimated values of $\theta_b(n)$ are given in Table~\ref{asymptotic_estimates_fig}.
\begin{table}
\[
\begin{array}{l|lllllllll}
b &2&3&4&5&6&7&8&9&10\\
\hline
 \theta_b(n) &2&1.11111&0.47166&0.21994&0.10853&0.05414&0.02706&0.01353&0.00676\\
3.46 (1/2)^{b-1}&1.73&0.86500&0.43250&0.21675&0.10813&0.05406&0.02703&0.01352&0.00676\\%&0.00338\\
\hbox{Error} & 0.27 &0.24611&0.03916&0.00319& 0.00040 &0.00008&0.00003&0.00001&0.000005 \\
\end{array}
\]
\caption{Table of rounded and estimated values of $\theta_b(n)$}
\label{asymptotic_estimates_fig}
\end{table}

\section{Conclusions}\label{conclusion}
We conclude with some questions and comments.
\begin{enumerate}
\item Linear recurrences on supporting domino towers, domino stacks, and right- or left-skewed domino towers have been given.  Can a direct linear recurrence on $c_b(n)$, the number of convex domino towers, be found?
\item If each domino block has uniform mass, a natural weight function on a domino tower can be defined by the location of the center of mass.  Which properties of domino towers have nice enumerative results?  For example, what are the enumerative results on weight zero towers or balanced towers, that is, towers whose center of mass is over the base?
\item Can column-convex or row-convex domino towers be counted?  What about domino towers without any ``holes"?
\item In chemistry, molecules are not built with uniform size, so it may be of interest to study towers built with horizontal pieces of mixed lengths, for example, towers utilizing horizontal pieces of lengths from a set $S \subset \mathbb{N}_{+}$.
\item The domino towers studied here are equivalent to two-dimensional polyominoes.  If we allow blocks with greater width or allow different orientations in three-dimensional space, the counting becomes much more difficult.  See Durhuss and Eilers~\cite{Durhuss_Eilers} for an exploration of towers of $2\times 4$ LEGO\textsuperscript{\textregistered} blocks. 
\end{enumerate}

%\noindent\rule[0.5ex]{\linewidth}{.1pt}
%
%\noindent \textit{2010 Mathematics Subject Classification:} Primary 05A15; Secondary  05A16.  
%
%\noindent \textit{Keywords:} domino, dissection problem, convex polyomino
%
%\noindent\rule[0.5ex]{\linewidth}{.1pt}
%
%\noindent (Concerned with sequences A034296, A065446, A117468, A275204, A275599 and A275662)

\end{document}